\documentclass[12pt]{article}
\usepackage{amssymb,amsmath,amsthm,secdot}
\DeclareMathOperator{\I}{I}%

\DeclareMathOperator{\Law}{Law}%

\def\E{\hskip.15ex\mathsf{E}\hskip.10ex}
\def\P{\mathsf{P}}

\def\eps{\varepsilon}
\def\phi{\varphi}

\newtheorem{Theorem} {Theorem}[section]
\newtheorem{Lemma}[Theorem] {Lemma}

\theoremstyle{definition}\newtheorem{Example}[Theorem]{Example}
\theoremstyle{definition}\newtheorem{Remark}[Theorem]{Remark}

\theoremstyle{definition}

\numberwithin{equation}{section}

\renewcommand{\ge}{\geqslant}
\renewcommand{\le}{\leqslant}

\newcommand{\sbt}{\,\begin{picture}(-1,1)(-1,-3)\circle*{1.5}\end{picture}\ }

\makeatletter
\renewcommand*{\@fnsymbol}[1]{\ensuremath{\ifcase#1\or *\or   \ddagger\or
   \mathsection\or \mathparagraph\or \|\or **\or \dagger\dagger
   \or \ddagger\ddagger \else\@ctrerr\fi}}
\makeatother

\title{On ergodic properties of nonlinear Markov chains and stochastic McKean--Vlasov equations}
\date{November 25, 2013}
\author{Oleg Butkovsky\thanks{Technion - Israel Institute of Technology,
Faculty of Industrial Engineering and Management, Haifa, 32000, Israel.}$\,\,{}^,$\thanks{
Lomonosov Moscow State University, Faculty of Mathematics and Mechanics, Department of Probability Theory, Moscow, 119991, Russia.}
${}^,$\thanks{Email: oleg.butkovskiy@gmail.com.}$\,\,{}^,$\thanks{Supported in part by Russian Foundation for Basic Research Grant 10-01-00397-a, Israel Science Foundation Grant 497/10, and a Technion fellowship.}}

\begin{document}
\maketitle

\begin{abstract}
We study ergodic properties of nonlinear Markov chains and stochastic McKean--Vlasov equations. For
nonlinear Markov chains we obtain sufficient conditions for existence and uniqueness of an invariant measure and uniform ergodicity. We also prove optimality of these conditions. For stochastic McKean--Vlasov equations we establish exponential convergence of their solutions to stationarity in the total variation metric under Veretennikov--Khasminskii-type conditions.



\end{abstract}

\section{Introduction}

In this paper we investigate ergodic properties of nonlinear Markov processes with discrete time (nonlinear
Markov chains) and ergodic properties of solutions of nonlinear stochastic differential equations (stochastic
McKean--Vlasov equations).

Recall that nonlinear Markov processes are stochastic processes whose transition functions may depend not only
on the current \textit{state} of the process but also on the current \textit{distribution} of the process. These
processes were introduced by H.P.~McKean \cite{Mc} to model plasma dynamics. Later nonlinear Markov processes
were studied by a number of authors, let us mention here the books of V.N.~Kolokoltsov \cite{Kol} and of
A.-S.~Sznitman \cite{Szn}. These processes arise naturally in the study of the limit behavior of a large number
of weakly interacting Markov processes (\cite{CGM}, \cite{Ganz}, \cite{Ver}) and have a wide range of
applications, including financial mathematics, population dynamics, neuroscience (see, e.g., \cite{Fr} and the
references therein).

As shown in \cite{MV}, nonlinear Markov processes may have peculiar ergodic properties. For instance, an
irreducible nonlinear Markov chain may have infinitely many invariant measures. Clearly, for standard
homogeneous Markov chains this is impossible (\cite[Proposition~10.1.1 and Theorem~10.4.9]{MT}).

This paper extends the recent work of the author \cite{myself} and consists of two parts. The first part of the
paper (Section~\ref{sect:2}) is devoted to ergodic properties of nonlinear Markov chains. We establish
sufficient conditions for existence and uniqueness of an invariant measure and uniform ergodicity of a nonlinear
Markov process. These conditions are optimal in a certain sense. It is interesting to note that in contrast to
the Markovian case, positivity of the elements of the one-step transition matrix does not imply even weak
convergence to the invariant measure, see Example~\ref{ex:remove} below.

The second part of the paper (Section~\ref{sect:3}) deals with stochastic McKean--Vlasov equations (SMVEs).
Recall that SMVE is a stochastic differential equation (SDE) whose drift and diffusion coefficients may depend
on the current distribution of the process. To study convergence of solutions of SMVE to an invariant measure
one usually considers associated nonlinear partial differential equation or investigates the mean-field limit. Using
and developing these ideas, P.~Cattiaux, A.~Guillin and F.~Malrieu (\cite{CGM}) and A.~Ganz (\cite{Ganz})
estimated convergence rate of strong solutions of SMVE to an invariant measure in the Wasserstein metric.
However these methods can not be applied to obtain such estimates in the total variation metric (recall that
this metric is stronger than the Wasserstein metric). To study convergence in total variation we
develop a new approach. This approach uses the ideas of M.~Hairer and J.~Mattingly \cite{HM}.

\section{Convergence of nonlinear Markov chains}\label{sect:2}

First of all, let us introduce some notation. We assume that all random objects are defined on a common
probability space $(\Omega,\mathcal{F},\P)$. Consider a measurable space $(E,\mathcal{E})$ and let $\mathcal
P(E)$ be the class of all probability measures on this space. The Dirac delta measure concentrated at a point $x\in
E$ is denoted by $\delta_x$.

Let $X=\left(X^\mu_n\right)_{n\in\mathbb{Z}_+}$ be a nonlinear Markov process with the state space
$(E,\mathcal{E})$, initial distribution $\Law(X^\mu_0)=\mu$, $\mu\in \mathcal P(E)$ and transition probabilities
\begin{equation*}
\P\bigl(X^\mu_{n+1}\in B|X^\mu_n=x\bigr)=P_{\mu_n}(x,B),\quad n\in\mathbb{Z}_+,
\end{equation*}
where $x\in E$, $B\in\mathcal{E}$, $n\in\mathbb{Z}_+$ and $\mu_n:=\Law(X^\mu_n)$. Further, by $\Law_{\mathsf{Q}}
\xi$ we denote the distribution of the random vector $\xi$ under the measure $\mathsf{Q}$. If $\mathsf{Q}=\P$,
then the subscript is omitted.

Note that if the function $P_\nu(x,B)$ does not depend on the measure $\nu$, then the process $X$ is
Markov (in this case the transition probability of $X$ is denoted by $P(x,B)$ and the argument $\nu$ is
dropped).

For probability measures $\mu,\nu\in\mathcal{P}(E)$ and a measurable function $f\colon E\to[0,\infty)$,
introduce the \textit{weighted total variation distance} by the following formula:
\begin{equation*}
d_{f}(\mu,\nu):=\sup_ {g:\,\, |g|\le f}\int_E g(x)(\mu(dx)-\nu(dx)).
\end{equation*}
In particular, if the function $f$ is identically equal to $1$, then the weighted total variation distance
coincides with the (unweighted)\textit{ total variation distance}; the latter is denoted by $d_{TV}$:
\begin{equation*}
d_{TV}(\mu,\nu):=2\sup_{A\in\mathcal{B}(E)} |\mu(A)-\nu(A))|, \quad\mu,\nu\in\mathcal{P}(E).
\end{equation*}

For a transition kernel $Q\colon E\times \mathcal{E}\to [0,1]$, a measurable function $\phi\colon
E\to\mathbb{R}$, and a probability measure $\nu\in\mathcal{P}(E)$, define
\begin{equation*}
Q\phi(x):=\int_E \phi(t) \,Q(x,dt);\quad Q \nu(dx):=\int_E Q(t,dx)\,\nu(dt);\quad \nu(\phi):=\int_E
\phi(t)\,\nu(dt).
\end{equation*}

We say that a transition kernel $Q$ satisfies the \textit{local Dobrushin condition} on a set $A\subset E$ if
there exists $\alpha>0$ such that for any $x,y\in A$
\begin{equation}\label{osnMP}
d_{TV} (Q(x,\cdot),Q(y,\cdot))\le 2(1-\alpha).
\end{equation}
If the kernel $Q$ satisfies the local Dobrushin condition on the whole space $E$, then we say that $Q$ satisfies
the \textit{global} Dobrushin condition.

A process $X$ is called \textit{uniformly ergodic} (see, e.g., \cite[p.~393]{MT}) if it has a stationary
distribution $\pi$ and there exist $C>0$, $\theta>0$ such that
\begin{equation*}
\sup_{\mu\in\mathcal P(E)} d_{TV}(\mu_n,\pi)\le C e^{-\theta n},\quad n\in\mathbb{Z}_+.
\end{equation*}
Recall that we denoted $\mu_n=\Law(X^\mu_n)$.

If the process $X$ is Markov, then the global Dobrushin condition is sufficient for the existence and uniqueness
of an invariant measure (\cite{Dobr}). Moreover, this condition implies uniform ergodicity of $X$ and guarantees
the following convergence rate: for any $\mu,\nu\in \mathcal P(E)$ one has \vspace{-1ex}
\begin{equation}\label{convrateMP}
d_{TV}(\mu_n,\nu_n)\le 2 (1-\alpha)^n,\quad n\in\mathbb{Z}_+.
\end{equation}

The following condition is a natural generalization of the global Dobrushin condition for nonlinear Markov
processes: there exists $\alpha>0$ such that for any $x,y\in E$
\begin{equation}\label{osnNMP}
\sup\limits_{\mu,\nu\in\mathcal P(E)}d_{TV} \bigl(P_\mu(x,\cdot),P_\nu(y,\cdot)\bigr)\le 2(1-\alpha).
\end{equation}
However, it turns out that in contrast to the Markov case, for any $0<\alpha<1$ this condition may be
insufficient even for a weak convergence of $\mu_n$ to the stationary measure. Let us give a corresponding
example.

\begin{Example}\label{ex:remove}
Let $X$ be a nonlinear Markov chain taking values in the state space $(E,\mathcal{E})=(\{1,2\},2^{\{1,2\}})$.
Define the transition probability matrix of the chain by the following formula:
\begin{equation*}
{\boldsymbol{P}_\nu} =
\begin{pmatrix}
 \bigl(\nu(\{2\})\wedge(1-\gamma/2)\bigr)\vee \gamma/2& \,\,\,\,\bigl(\nu(\{1\})\wedge(1-\gamma/2)\bigr)\vee \gamma/2\\[1.5ex]
 \bigl(\nu(\{2\})\wedge(1-\gamma/2)\bigr)\vee \gamma/2& \,\,\,\,\bigl(\nu(\{1\})\wedge(1-\gamma/2)\bigr)\vee \gamma/2
\end{pmatrix}
\,,
\end{equation*}
where $\gamma\in(0,1)$. Here, as usual, $a\wedge b=\min\{a,b\}$ and ${a\vee b=\max}\{a,b\}$ for real $a,b$. It
is clear that this nonlinear Markov chain satisfies condition \eqref{osnNMP} with $\alpha=\gamma$ and has the
stationary distribution $\pi:=(\delta_1+\delta_2)/2$. On other hand, for any $a\in[\gamma/2, 1-\gamma/2]$,
$a\neq1/2$, and initial distribution $\mu_0 (a):=a \delta_1  + (1-a) \delta_2$ the measure $\mu_n=\mu_n(a)$ does
not converge to $\pi$ as $n\to\infty$. Indeed, if $n$ is even, then $\mu_n(a)=a\delta_1 + (1-a) \delta_2$,
whereas for odd $n$ one has $\mu_n(a)=(1-a) \delta_1  + a \delta_2$.
\end{Example}

Thus, the global Dobrushin condition \eqref{osnNMP} does not guarantee uniform ergodicity of a nonlinear Markov
process. Furthermore, as explained below, \eqref{osnNMP} does not imply even existence of a stationary measure.
Let us show how this condition can be extended in such a way that the new condition is sufficient for uniform
ergodicity.

\begin{Theorem}\label{Th:DiscreteTime}
Assume that the process $X$ satisfies condition \eqref{osnNMP} for some $\alpha>0$.

\textbf{$($i$)$} If there exists $\lambda\in[0,\alpha]$ such that for all $x\in E$ and $\mu,\nu\in\mathcal P(E)$
one has
\begin{equation}\label{cond2}
d_{TV}\bigl(P_\mu(x,\cdot),P_\nu(x,\cdot)\bigr)\le\lambda\;\! d_{TV}(\mu,\nu),
\end{equation}
then the process $X$ has a unique invariant measure $\pi$. Moreover, if $\lambda<\alpha$, then for any measure $\mu\in\mathcal{P}(E)$
\begin{equation}\label{convrate_fast}
d_{TV}(\mu_n,\pi)\le2 (1-(\alpha-\lambda))^n,\quad n\in\mathbb{Z}_+,
\end{equation}
and if $\lambda=\alpha$, then
\begin{equation}\label{convrate_slow}
d_{TV}(\mu_n,\pi)\le 2/(\lambda n),\quad n\in\mathbb{Z}_+.
\end{equation}
\vskip1ex \textbf{$($ii$)$} The restriction $\lambda\in[0,\alpha]$ in condition \eqref{cond2} is optimal.
Namely, for any pair $(\alpha,\lambda)$ such that ${0<\alpha<\lambda\le 1}$  there exist processes $X=\!\left(X_n\right)_{n\in\mathbb{Z}_+}$,
$Y=\!\left(Y_n\right)_{n\in\mathbb{Z}_+}$, $Z=\!\left(Z_n\right)_{n\in\mathbb{Z}_+ }$, that satisfy conditions
\eqref{osnNMP} and \eqref{cond2}, and measure $\mu\in\mathcal P(E)$, such that the process $X$ has more than
one stationary measure, the process $Y$ has no stationary measures, $d_{TV}(Z_n^\mu,\pi)\not\to 0$ as
$n\to\infty$.

\end{Theorem}

The proof of Theorem \ref{Th:DiscreteTime} is given in Section \ref{S:proofs}.

\begin{Remark}\label{rm:M}
If the process $X$ is Markov, then condition \eqref{osnNMP} coincides with \eqref{osnMP}, and condition
\eqref{cond2} is  satisfied with $\lambda=0$. Moreover, the rate of convergence provided by
\eqref{convrate_fast} coincides with the corresponding rate of convergence for Markov processes, as formula
\eqref{convrateMP} shows. Thus, Theorem~\ref{Th:DiscreteTime} extends the classical result of Dobrushin
\cite{Dobr}.
\end{Remark}

Now we pass on to the study of nonlinear processes that do not satisfy the  \textit{global} Dobrushin condition,
but satisfy the \textit{local} Dobrushin condition on a certain ``good'' set.

\section{McKean--Vlasov equation with small perturbation}\label{sect:3}

Consider a multidimensional stochastic McKean--Vlasov equation in $\mathbb{R}^d$, $d\ge1$
\begin{equation}\label{SMV}
\left\{\begin{array}{l}
X_t=X_0+\int_0^t b(X_s,\mu_s)\,ds+W_t,\quad t\ge0,\\
\Law(X_t)=\mu_t,
\end{array}\right.
\end{equation}
where $b\colon \mathbb{R}^d\times \mathcal{P}(\mathbb{R}^d)\to\mathbb{R}^d$, $W$ is a  $d$-dimensional Brownian
motion, and initial condition $X_0$ is a $d$-dimensional vector that is independent of $W$.

We say that a function $h\colon\mathbb{R}^d\to\mathbb{R}^{d}$ satisfies the \textit{Veretennikov--Khasminskii
condition} if there exist $M>0$, $r>0$ such that
\begin{equation}\label{VH}
\langle h(x),x\rangle\le  - r |x|,\quad x\in\mathbb{R}^d,\,|x|\ge M.
\end{equation}
Here $\langle\cdot,\cdot\rangle$ is a standard scalar product in $\mathbb{R}^d$.

If the drift coefficient $b$ does not depend on the measure  $\mu$ (and satisfies certain conditions), then
\eqref{SMV} is a stochastic differential equation and its strong solution is a Markov process. Ergodic
properties of this Markov process were studied by many authors.  As shown in \cite{V1987}, if the function  $b$
satisfies inequality \eqref{VH}, then the strong solution of this equation has a unique invariant
measure. Moreover, $\Law(X_t)$ converges exponentially to this measure in total variation as  $t\to\infty$.

Let us extend this result to the case of McKean--Vlasov equations. Assume that the drift $b$ consists of two
parts $b_1$ and $\eps b_2$, where the function $b_1$ does not depend on the measure, and $\eps b_2$ is a small nonlinear perturbation. In other words,
\begin{equation}\label{snos}
b(x,\mu)=b_\eps(x,\mu)=b_1(x)+\eps b_2(x,\mu),\quad x\in\mathbb{R}^d,\,\, \mu\in\mathcal{P}(\mathbb{R}^d),
\end{equation}
where $\eps>0$. We also assume that the functions $b_1$ and $b_2$ are Lipschitz, i.e.,
there exists a positive $L>0$ such that
\begin{equation}\label{lip}
|b_1(x)-b_1(y)|+|b_2(x,\mu)-b_2(y,\nu)|\le L(|x-y|+\rho_2(\mu,\nu)),\,\, x,y\in\mathbb{R}^d,\,\,
\mu,\nu\in\mathcal{P}(\mathbb{R}^d),
\end{equation}
where $\rho_2(\mu,\nu)$ is the 2-Wasserstein distance between the measures $\mu$ and $\nu$. Recall that this distance is defined by the following formula
\begin{equation*}
\rho_2(\mu,\nu):=\Bigl(\inf_{\lambda\in \mathcal{C}(\mu,\nu)} \int_{\mathbb{R}^d\times \mathbb{R}^d}
|x-y|^2\wedge 1 \,\lambda(dx,dy)\Bigr)^{1/2},
\end{equation*}
where $\mathcal{C}(\mu,\nu)$ is the set of all probability measures on
$(\mathbb{R}^d\!\!\:\times\!\!\: \mathbb{R}^d)$ with marginals $\mu$ and~$\nu$.

As shown in \cite[Proposition~1.2]{JMW}, under these conditions for any $\eps\ge0$ equation \eqref{SMV} has a unique strong solution $(X_t^\eps,\mu_t^\eps)_{t\ge 0}$. If the initial distribution $\mu=\mu_0^\eps$ is fixed (and hence the distribution $\mu_t^\eps$ is also fixed), then the process $(X_t^\eps)_{t\ge 0}$
is a nonhomogeneous Markov process (however its transition probabilities are different for different initial distributions $\mu$). We will denote $P^t_\mu(\eps)(x,A):=\P_{0,x}(X_t^\eps\in A)$, where $\eps>0$, $x\in
\mathbb{R}^d$ and $A\in\mathcal{B}(\mathbb{R}^d)$. By definition, we have $\mu_t^\eps=P^t_{\mu}(\eps)\mu$.

\begin{Theorem}\label{Th:mv}
Suppose that conditions \eqref{snos} and \eqref{lip} are satisfied. Assume additionally that
\begin{enumerate}
    \item [1)] the function $b_1$ satisfies condition \eqref{VH};

    \item [2)] the function $b_2$ is uniformly bounded, i.e. $$\sup_{\substack{x\in\mathbb{R}^d\\\mu\in\mathcal{P}(\mathbb{R}^d)}}|b_2(x,\mu)|\le D$$
        for some $D>0$.
\end{enumerate}

Then there exists $\eps_0>0$ such that for any $\eps\in [0,\eps_0]$ McKean-Vlasov equation \eqref{SMV} has a unique invariant measure $\pi^\eps$. Moreover, for any measure $\mu_0\in\mathcal{P}(\mathbb{R}^d)$ such that $I(\mu_0):=\int_{\mathbb{R}^d} e^x \,\mu_0(dx)<\infty$ one has
\begin{equation}\label{otsenka}
d_{TV}(\mu_t^\eps,\pi^\eps)\le C (1+I(\mu_0))e^{-\theta t},\quad t\ge0,
\end{equation}
for some positive $C=C(\eps)$ and $\theta=\theta(\eps)$.
\end{Theorem}

Thus, Theorem~\ref{Th:mv} shows that if a drift that satisfies Veretennikov-Khasminskii condition undergoes a small nonlinear perturbation, then ergodic properties of the strong solution of McKean--Vlasov equation remain the same.

While the proof of the theorem is postponed to Section~\ref{S:proofs}, we outline now the main steps. Note  that for a fixed $\eps>0$ the discretized process $(X_{Tn}^\eps)_{n\in\mathbb{Z}_+}$, where $T>0$, is a nonlinear Markov chain. However this process satisfies condition \eqref{osnNMP} only on certain subsets of $\mathbb{R}^d$ (for instance, on all compact sets) but not on the whole space $\mathbb{R}^d$. Therefore Theorem~\ref{Th:DiscreteTime} cannot be applied here.

\begin{proof}[Sketch of the proof of Theorem~\ref{Th:mv}]
To prove the theorem we develop the Hairer--Matt\-in\-gly technique (\cite{Hair}, \cite{HM}) of constructing auxiliary contraction mappings. Namely, we show that for all sufficiently ``small'' $\eps$  and a certain ``nice'' function $f\colon\mathbb{R}^d\to \mathbb{R}_+$ we have the following contraction inequality in the weighted total variation metric $d_f$:
\begin{equation*}
d_f(P^1_\mu (\eps)\mu, P^1_\nu(\eps) \nu)\le \lambda d_f(\mu, \nu),\quad\mu,\nu\in\mathcal{P}(\mathbb{R}^d)
\end{equation*}
where $\lambda\in(0,1)$ and $\lambda$ does not depend on measures $\mu,\nu$. This inequality yields the existence and uniqueness of an invariant measure as well as exponential convergence to stationarity.
\end{proof}

\section{Proofs}\label{S:proofs}
\begin{proof}[Proof of Theorem $\ref{Th:DiscreteTime}$]
\textbf{$($i$)$} First of all, note that for any probability measures $\mu,\nu\in\mathcal{P}(E)$ one has
\vskip-1ex
\begin{equation}\label{szhatie}
d_{TV}(P_\mu \mu,P_\nu \nu)\le d_{TV}(\mu,\nu) (1-\alpha+\lambda)-\lambda d_{TV}(\mu,\nu)^2/2.
\end{equation}
Indeed, if $d_{TV}(\mu,\nu)=0$, then $\mu=\nu$ and inequality \eqref{szhatie} is trivial. Otherwise, denote
$\eta(dx):=(d\mu/d\nu\wedge 1)\nu(dx)$, where $d\mu/d\nu$ is a Radon--Nikodym derivative of absolutely continuous part of $\mu$ with respect to $\nu$. Then
\begin{equation}\label{vspom}
d_{TV}(P_\mu \mu,P_\nu \nu)\le d_{TV}(P_\mu \eta,P_\nu \eta) + d_{TV}(P_\mu (\mu-\eta),P_\nu (\nu-\eta)).
\end{equation}
Applying inequality \eqref{cond2}, we get
\begin{equation}\label{vspom2}
d_{TV}(P_\mu \eta,P_\nu \eta)\!\le2\!\int_{E\times E}\bigl|P_\mu(x,dt)-P_\nu(x,dt)\bigr|\eta(dx)\!\le \lambda d_{TV}(\mu,\nu) (1-d_{TV}(\mu,\nu)/2),
\end{equation}
where we took into account that $\eta(E)=1-d_{TV}(\mu,\nu)/2$.

On the other hand, it follows from \eqref{osnNMP} that
\begin{align*}
d_{TV}&(P_\mu (\mu-\eta),P_\nu
(\nu-\eta))=2\int_E\Bigl|\int_E P_\mu(x,dt)(\mu-\eta)(dx)-\int_E P_\nu(y,dt)(\nu-\eta)(dy)\Bigr|\\
&\le\frac{4}{d_{TV}(\mu,\nu)}\int_{E^3}|P_\mu(x,dt)-P_\nu(y,dt)|(\mu-\eta)(dx)(\nu-\eta)(dy)\le
(1-\alpha)d_{TV}(\mu,\nu).
\end{align*}
Combining this inequality with \eqref{vspom} and \eqref{vspom2}, we derive
\eqref{szhatie}.

By iterating inequality \eqref{szhatie} $n$ times, we obtain that if $\lambda<\alpha$, then
\begin{equation}\label{epsgl}
d_{TV}(P^n_\mu \mu,P^n_\nu \nu)\le 2(1-\alpha+\lambda)^n,
\end{equation}
and if $\lambda=\alpha$, then
\begin{equation}\label{epsrl}
d_{TV}(P^n_\mu \mu,P^n_\nu \nu)\le 2/(\lambda n).
\end{equation}

Now we prove that the process $X$ has an invariant measure. Consider the sequence of measures $(\mu_n)_{n\in\mathbb{Z}_+}$. Let us verify that this sequence is a Cauchy sequence in the metric space $(\mathcal{P}(E),d_{TV})$. It follows from \eqref{epsgl} and \eqref{epsrl} that for any  $m,n\in\mathbb{N}$ we have
$$
d_{TV}(\mu_n ,\mu_{n+m})=d_{TV}(P^n_{\mu_0} \mu_0,P^n_{\mu_m} \mu_m)\le 2/(\lambda
n).
$$
Since the space $(\mathcal{P}(E),d_{TV})$ is complete, we see that there exist a probability measure $\pi\in
\mathcal{P}(E)$ such that $d_{TV}(\mu_n,\pi)\to 0$ as $n\to\infty$. We make use of \eqref{szhatie} to derive for any positive integer  $n$
\begin{align*}
d_{TV}(P_\pi \pi,\pi)&\le d_{TV}(P_\pi \pi,\mu_{n+1})+d_{TV}(\mu_{n+1},\pi)=
d_{TV}(P_\pi \pi,P_{\mu_n} \mu_n)+d_{TV}(\mu_{n+1},\pi)\\
&\le d_{TV}(\pi,\mu_n)+d_{TV}(\pi,\mu_{n+1}).
\end{align*}
The right-hand side of the above inequality tends to $0$ as $n\to\infty$. Hence  $P_\pi \pi =\pi$ and the measure $\pi$ is invariant. The uniqueness of an invariant measure follows from \eqref{szhatie}. Indeed, if a measure $\zeta$ is also invariant for the process $X$ and $\pi\neq\zeta$, then
$d_{TV}(\pi,\zeta)=d_{TV}(P_{\pi} \pi,P_{\zeta} \zeta)< d_{TV}(\pi,\zeta)$, which is impossible.

Substituting $\pi$ for $\nu$ in \eqref{epsgl} and \eqref{epsrl}, we get
\eqref{convrate_fast} and \eqref{convrate_slow}.

\textbf{$($ii$)$} Now let us prove that the condition $\lambda\le\alpha$ is optimal. Assume that this condition does not hold and $0<\alpha<\lambda\le 1$. First, we give an example of the process that has more than one stationary measure.

Consider a nonlinear Markov chain $X$ taking values in the state space $(E,\mathcal{E})=(\{1,2\},2^{\{1,2\}})$. Define the transition probability matrix of the chain by the following formula:
\begin{equation*}
\!{\boldsymbol{P_\nu}} =\!
\begin{pmatrix}
 \bigl(\!(1-\lambda\nu(\{2\}))\wedge(1-\frac{\alpha}{2})
 \bigr)\vee (1-\lambda+\frac{\alpha}{2})&  \bigl(\lambda\nu(\{2\})\wedge(\lambda-\frac{\alpha}{2})
 \bigr)\vee \frac{\alpha}{2}\\[0.5ex]
 \bigl(\lambda\nu(\{1\})\wedge(\lambda-\frac{\alpha}{2})
 \bigr)\vee \frac{\alpha}{2}&\!\!\!\! \bigl((1-\lambda\nu(\{1\}))\wedge(1-\frac{\alpha}{2})
 \bigr)\vee (1-\lambda+\frac{\alpha}{2})
\end{pmatrix}\!
.\end{equation*}

Let us verify that $X$ satisfies conditions \eqref{osnNMP} and \eqref{cond2}. Indeed, for any $i,j\in E$ and $\mu,\nu\in\mathcal P(E)$ we have  $P_\mu(i,j)\ge\alpha/2$. Consequently,
\begin{equation*}
P_\mu(i,1)\wedge P_\nu(j,1)+P_\mu(i,2)\wedge P_\nu(j,2)\ge \alpha,
\end{equation*}
and condition\eqref{osnNMP} holds. Similarly, for any $i\in E$, $\mu,\nu\in\mathcal P(E)$
\begin{equation*}
\bigl|P_\mu(i,1)-P_\nu(i,1)\bigr|+\bigl|P_\mu(i,2)-P_\nu(i,2)\bigr|\le \lambda d_{TV}(\mu,\nu),
\end{equation*}
and condition \eqref{cond2} is also met.

On the other hand, it is easy to see that for any $a\in I:=[\alpha/(2 \lambda), 1-\alpha/(2\lambda)]$ the measure $\mu(a):=a\delta_1+(1-a)\delta_2$ is stationary for $X$. Hence, $X$ has more than one stationary measure (actually, continuum of stationary measures). Moreover, $d_{TV}(X_n^{\mu(a_1)},X_n^{\mu(a_2)})\not\to 0$ as $n\to\infty$ for $a_1,a_2\in I$, $a_1\neq a_2$.

Now we give an example of the process, which has no stationary measures. To construct the example one should
consider processes taking values in infinite state spaces. Indeed, if a process has a finite state space and satisfies \eqref{osnNMP} and \eqref{cond2}, then, by Brouwer fixed-point theorem, it has (at least one) stationary measure.

Let $X$ be a nonlinear Markov chain with the state space $(E,\mathcal{E})=(\mathbb{N},2^{\mathbb{N}})$.
Define the transition probabilities $P_\nu(i,j)$,\, $i,j\in\mathbb{N}$, $\nu\in \mathcal P(\mathbb{N})$ by the following formulas:
\begin{align*}
P_\nu(i,1)&:=\bigl(\lambda \nu(\{1\})\bigr)\vee\alpha\\
P_\nu(i,j)&:=\bigl((\lambda \nu(\{1,2,\dots,j\})-\alpha)\wedge \lambda \nu(\{j\})\bigr)\vee 0\, +
(1-\lambda)\I(j=i+1),\quad j\neq 1.
\end{align*}

First let us show that the transition probabilities are well-defined, i.e. their sum in each row is $1$. Fix $\nu\in \mathcal P(\mathbb{N})$. Since $\lambda>\alpha$, we see that there exists a positive integer $n=n(\nu)$ such that $\lambda \nu(\{1,2,\dots,n\})\ge\alpha$ and $\lambda \nu(\{1,2,\dots,n-1\})<\alpha$. If $n=1$, then $P_\nu(i,j)=\lambda \nu(\{j\})+(1-\lambda)\I(j=i+1)$ for all $i,j\in\mathbb{N}$, and the sum over $j$ of transition probabilities is obviously $1$. If $n>1$, then for any $i\in\mathbb{N}$ we get
\begin{align*}
P_\nu(i,1)&=\alpha,\\
P_\nu(i,j)&=(1-\lambda)\I(j=i+1),\quad\text{for }1<j<n,\\
P_\nu(i,n)&=\lambda \nu(\{1,2,\dots,n\})-\alpha+(1-\lambda)\I(n=i+1),\\
P_\nu(i,j)&=\lambda \nu(\{j\})+(1-\lambda)\I(j=i+1)\quad\text{for }j>n.
\end{align*}
Therefore in this case also $\sum_{j=1}^\infty P_\nu(i,j)=1$.

Since for any positive integer $i$ and for any measure $\nu\in \mathcal P(\mathbb{N})$ we have $P_\nu(i,1)\ge\alpha$, we see that the process $X$ satisfies \eqref{osnNMP}. Let us verify that condition \eqref{cond2} also holds. Fix $i\in\mathbb{N}$ and measures $\mu,\nu\in\mathcal P(\mathbb{N})$. Then the left-hand side of \eqref{cond2} is equal to $\sum_{j=1}^\infty \bigl|P_\nu(i,j)-P_\mu(i,j)\bigr|$. Define $n(\nu)$, $n(\mu)$ as above. Without loss of generality, suppose $n(\nu)\ge n(\mu)$. If, actually, $n(\nu)> n(\mu)$, then
\begin{align}\label{provlambda}
\sum\limits_{j=1}^\infty \bigl|P_\nu(i,j)-P_\mu(i,j)\bigr|=&
\lambda \mu(\{1,\dots,n(\mu)\})-\alpha+\sum\limits_{j=n(\mu)+1}^{n(\nu)-1} \lambda \mu(\{j\})\nonumber\\
&+\bigl|\lambda \nu(\{1,\dots,n(\nu)\})-\alpha-\lambda \mu(\{n(\nu)\})\bigr|\nonumber\\
&+\lambda\sum\limits_{j=n(\nu)+1}^\infty \bigl|\mu(\{j\})-\nu(\{j\})\bigr|.
\end{align}
By definition of $n(\nu)$, we have $\alpha-\lambda\nu(\{1,\dots,n(\nu)-1\})>0$. Hence
\begin{equation*}
\bigl|\lambda \nu(\{1,\dots,n(\nu)\})-\alpha-\lambda \mu(\{n(\nu)\})\bigr|\!\le\! \alpha-\lambda
\nu(\{1,\dots,n(\nu)-1\})+\lambda \bigl|\nu(\{n(\nu)\})- \mu(\{n(\nu)\})\bigr|.
\end{equation*}
Combining the last inequality with \eqref{provlambda}, we derive
\begin{align*}
\sum\limits_{j=1}^\infty \bigl|&P_\nu(i,j)-P_\mu(i,j)\bigr|\\
&\le \lambda \bigl(\mu(\{1,\dots,n(\nu)-1\})-\nu(\{1,\dots,n(\nu)-1\})\bigr)+\lambda\sum\limits_{j=n(\nu)}^\infty \bigl|\mu(\{j\})-\nu(\{j\})\bigr|\\
&\le \lambda\sum\limits_{j=1}^\infty\bigl|\mu(\{j\})-\nu(\{j\})\bigr|=\lambda d_{TV}(\mu,\nu),
\end{align*}
and condition \eqref{cond2} is satisfied. If $n(\nu)=n(\mu)$, then, by a similar argument
\begin{align*}
\sum\limits_{j=1}^\infty \bigl|&P_\nu(i,j)-P_\mu(i,j)\bigr|\\
&=\lambda\bigl|\nu(\{1,\dots,n(\nu)\})-\mu(\{1,\dots,n(\nu)\})\bigr|
+\lambda\sum\limits_{j=n(\nu)+1}^\infty \bigl|\mu(\{j\})-\nu(\{j\})\bigr|\\
&\le \lambda\sum\limits_{j=1}^\infty \bigl|\mu(\{j\})-\nu(\{j\})\bigr|=\lambda d_{TV}(\mu,\nu).
\end{align*}

Finally, let us verify that the process $X$ has no stationary measures. Assume the converse. Let the measure
$\mu\in\mathcal{P}(\mathbb{N})$ be invariant for $X$. Then for any $j\in\mathbb{N}$
\begin{equation}\label{std}
\sum\limits_{j=1}^\infty \mu(\{i\}) P_\mu(i,j)=\mu(\{j\}).
\end{equation}
If $\mu(\{1\})\ge\alpha/\lambda$, then \eqref{std} implies $\mu(\{1\})=0$. Therefore
$\mu(\{1\})<\alpha/\lambda$ and hence $\mu(\{1\})\!=\!\alpha$. Define $n(\mu)$ as above. It follows from the definition of  $n(\mu)$ that  ${\mu(\{n(\mu)\})\!>\!0}$. On the other hand, \eqref{std} yields
\begin{align*}
\mu(\{i\})&=\alpha (1- \lambda)^{i-1}\quad\text{ for } 1\le i<n(\mu),\\
\mu(\{n(\mu)\})&=\alpha(1- \lambda)^{n(\mu)-1}+\lambda\mu(\{1,\dots,n(\mu)\})-\alpha,
\end{align*}
whence $\mu(\{n(\mu)\})=0$. This contradiction proves that $X$ has no invariant measures.
\end{proof}

Now we move on to the proof of Theorem~\ref{Th:mv}. We will use the following lemma, which is due to M.~Hairer and J.~Mattingly.

\begin{Lemma}[{\cite[Theorem~3.9]{Hair}}, see also \cite{HM}]\label{L:2HM}
Let $Q$ be a Markov transition kernel on a measurable space $(E,\mathcal{E})$. Assume that for some function $V\colon E\to \mathbb{R}_+$ and constants $K\ge0$, $\gamma\in[0,1)$ one has
\vskip-2ex
\begin{equation}\label{lyap_gen}
Q V\le \gamma V + K.
\end{equation}
Furthermore, assume that the kernel $Q$ satisfies the local Dobrushin condition \eqref{osnMP} on the set $\{x\in E: V(x)\le
4K/(1-\gamma)\}$.

Then there exist constants $\lambda\in[0,1)$ and $\beta>0$ that depend only on
$\alpha$ from \eqref{osnMP}, $\gamma$ and $K$ $($but not on the kernel $Q${}$)$ such that
\begin{equation*}
d_{1+\beta V} (Q \mu,Q \nu )\le \lambda d_{1+\beta V}  (\mu,\nu)
\end{equation*}
for any $\mu,\nu\in\mathcal{P}(E)$.
\end{Lemma}

Consider the following auxiliary SDE:
\begin{equation}\label{sdu_igr}
d Y_t^{(x)}=b_1(Y_t^{(x)})\,dt+dW_t,\quad Y_0^{(x)}=x,
\end{equation}
where the function $b_1$ is defined in \eqref{snos}. The lemma below is well-known to the specialists working in this area. For example, one can find the statement of the lemma (without proof) in \cite[p.~317]{VerTVP} or in \cite[p.~603]{Kulik} and the idea of the proof (to apply the Harnack inequality) was suggested to the author by A.Yu.~Veretennikov and A.M.~Kulik. Nevertheless, the author were not able to find the full proof of the lemma in the literature. Therefore we give it here for the completeness  of exposition.

\begin{Lemma}\label{L:Lok_dobr}
If the function $b_1$ is Lipschitz, then the strong solution of SDE \eqref{sdu_igr} satisfies the local Dobrushin condition on any compact set. In other words, for any $t>0$ and $R>0$ there exists $\alpha=\alpha(R,t)>0$ such that
\begin{equation}\label{DobrForSDE}
d_{TV}(\Law(Y_t^{(x)}),\Law(Y_t^{(y)}))\le 2 (1-\alpha(R,t)),\quad |x|\le R,\,|y|\le R.
\end{equation}
\end{Lemma}

\begin{proof}

Let us use the Harnack inequality for diffusion processes \cite[Theorem~1.1(2)]{Wang}. This inequality can be written in the following form: for any $t>0$ there exists $C=C(t)>0$ such that
\begin{equation}\label{Harnak}
(\P(Y_t^{(y)}\in A))^2\le \P(Y_t^{(x)}\in A)\exp(C(t)|x-y|^2)
\end{equation}
for any $x,y\in\mathbb{R}^d$, $A\in\mathcal{B}(\mathbb{R}^d)$. We claim that for any
$A\in\mathcal{B}(\mathbb{R}^d)$ we have the following estimate
\begin{equation*}
|\P(Y_t^{(y)}\in A)-\P(Y_t^{(x)}\in A)|\le 1- \exp(-C(t)|x-y|^2)\wedge (1/2),\quad x,y\in\mathbb{R}^d.
\end{equation*}
Indeed, suppose without loss of generality that $\P(Y_t^{(y)}\!\in\! A)>\P(Y_t^{(x)}\!\in\! A)$. Then by
\eqref{Harnak}, we have
\begin{align*}
|\P(Y_t^{(y)}\in A)&-\P(Y_t^{(x)}\in A)|\\
&\le \P(Y_t^{(y)}\in A)-(\P(Y_t^{(y)}\in A))^2\exp(-C(t)|x-y|^2)\\
&\le \P(Y_t^{(y)}\in A)-(\P(Y_t^{(y)}\in A))^2(\exp(-C(t)|x-y|^2)\wedge (1/2))\\
&\le 1- \exp(-C(t)|x-y|^2)\wedge (1/2).
\end{align*}
Therefore
\begin{equation*}
d_{TV}(\Law(Y_t^{(x)}),\Law(Y_t^{(y)}))\le 2- 2\exp(-C(t)|x-y|^2)\wedge 1.
\end{equation*}
This inequality implies \eqref{DobrForSDE}.
\end{proof}

Let us introduce twice continuously differentiable function $V\colon\mathbb{R}^d\to \mathbb{R}_+$ such that
${V(x)=e^{|x|(r/4\wedge 1)}}$ for $|x|\ge M$. Let us check that for some $T>0$ and small $\eps$ the transition kernels $P^T_\mu(\eps)$, $\mu\in\mathcal{P}(\mathbb{R}^d)$ satisfy conditions \eqref{osnMP} and \eqref{lyap_gen}.

\begin{Lemma}\label{L:3Proverka}
Suppose the conditions of Theorem~$\ref{Th:mv}$ hold. Then there exist  $T>0$, $\eps_0>0$, $\alpha>0$, $K>0$,
and $\gamma\in[0,1)$ such that
\begin{align}\label{Lyap}
&P^T_\mu(\eps) V \le \gamma V+K,\\
\label{Lyap2} &P^T_\mu(\eps) V^2 \le \gamma^2 V^2+K^2,\\
\label{dobr} &d_{TV}(P^T_\mu(\eps) \delta_x,P^T_\mu(\eps) \delta_y)\le 2(1-\alpha),\quad x,y\in S_V
\end{align}
for any $\eps\!\in\![0,\eps_0]$ and any measure $\mu\!\in\!\mathcal{P}(\mathbb{R}^d)$. Here
${S_V\!:=\!\{x\!\in\!\mathbb{R}^d\!: V(x)\!\le\!4K/\!(1-\gamma)\}}$.

Moreover for any measure $\mu\in\mathcal{P}(\mathbb{R}^d)$ we have the following estimate
\begin{equation}\label{est_mera}
(P^t_\mu(\eps) \mu)(V)\le \mu(V)+K,\quad t\ge0.
\end{equation}
\end{Lemma}

\begin{proof}
First, let us prove inequality \eqref{Lyap2}. Denote $G(x):=V^2(x)$ and set $\kappa:=(r/4)\wedge
1$. It follows from the definition of $V$ that $G(x)=\exp(2\kappa|x|)$ for $|x|\ge M$. Fix a measure $\mu$ and let $Z^\eps=(Z_t^\eps)_{t\ge0}$ be a strong solution of the SDE
\begin{equation}\label{vspomogat}
d Z_t^\eps=b_\eps(Z^\eps_t,\mu^\eps_t)\,dt+dW_t,\quad Z_0^\eps=x,
\end{equation}
where  $(X_t^\eps,\mu_t^\eps)_{t\ge0}$ is a strong solution of SMVE \eqref{SMV} with the initial distribution ${\mu_0^\eps=\mu}$. The definition of the process $Z^\eps$ yields
$$
P^t_\mu(\eps) G(x)=\E_x G(Z_t^\eps),\quad x\in\mathbb{R}^d.
$$
Apply Ito's formula to the function $e^{\theta t}G(Z^\eps_t)$, $\theta \in\mathbb{R}$. We derive
\begin{align*}
e^{\theta t}& \E_x G(Z_t^\eps) - G(x)\\
&\le\E_x\int_0^t e^{\theta s}\I(|Z_s^\eps|\ge M) G(Z_s^\eps) \bigl(\theta+2\kappa^2+2\kappa\langle
Z_s^\eps,b_\eps(Z_s^\eps,\mu_s^\eps)\rangle|Z_s^\eps|^{-1}\bigr) \,ds+Ce^{\theta t} \\
&\le (\theta+2\kappa^2-2\kappa(r-\eps D))\E_x\int_0^t e^{\theta s}\I(|Z_s^\eps|\ge M) G(Z_s^\eps)
\,ds+Ce^{\theta t},
\end{align*}
where $C=C(\theta)>0$, and in the second inequality we made use of condition~\eqref{VH} and the boundedness of $b_2$. By taking $\theta=2\kappa(r-\eps D-\kappa)$, we obtain
\begin{equation*}
P^t_\mu(\eps) G(x)=\E_x G(Z_t^\eps) \le e^{-\kappa t r/2} G(x)+K^2,\quad x\in\mathbb{R}^d,\,\, t\ge0,\,\,
0\le\eps\le \frac{r}{2D},
\end{equation*}
where $K= C(\theta)^{1/2}$. This implies \eqref{Lyap2}. Moreover, by Jensen's inequality,
\begin{equation}\label{Vest}
P^t_\mu(\eps) V(x)\le (P^t_\mu(\eps) V^2(x))^{1/2} \le (e^{-\kappa t r/2} V^2(x)+K^2)^{1/2}\le e^{-\kappa t r/4} V(x)+K
\end{equation}
for any $0\le \eps\le r/(2D)$. This yields \eqref{Lyap}.

To prove inequality \eqref{dobr} we consider SDE \eqref{sdu_igr}. We choose sufficiently large $R$ so that
$S_V\subset\{x\in\mathbb{R}^d: |x|\le R\}$. We make use of \eqref{DobrForSDE} to derive
\begin{align}\label{dtvest}
d_{TV}(P^t_\mu(\eps) \delta_x,P^t_\mu(\eps) \delta_y)\le& d_{TV}(P^t_\mu(\eps) \delta_x,\Law(Y_t^{(x)}))+ d_{TV}(\Law(Y_t^{(x)}),\Law(Y_t^{(y)}))\nonumber\\
&+d_{TV}(\Law(Y_t^{(y)}),P^t_\mu(\eps) \delta_y)\nonumber\\
\le& d_{TV}(P^t_\mu(\eps) \delta_x,\Law(Y_t^{(x)})) + d_{TV}(\Law(Y_t^{(y)}),P^t_\mu(\eps) \delta_y)\nonumber\\
&+2(1-\alpha(R,t))
\end{align}
for any $x,y\in S_V$, $t>0$.

Introduce the probability measure  $\P^{\eps,\mu}$ on $(\Omega,\mathcal{F})$ by putting
\begin{equation}\label{mera}
\frac{d {\P^{\eps,\mu}}}{d \P}:=\exp\Bigl(\eps\int_0^t
b_2(Y^{(x)}_s,\mu_s^\eps)\,dW_s-\frac12\eps^2\int_0^t
|b_2(Y^{(x)}_s,\mu_s^\eps)|^2\,ds\Bigr).
\end{equation}
Since the function $b_2$ is bounded, we see that the measure $\P^{\eps,\mu}$ is well-defined. By the Girsanov theorem, the process
\vspace{-1ex}
\begin{equation}\label{newbd}
W_s^{\eps,\mu}:=W_s-\eps\int_0^s b_2(Y^{(x)}_u,\mu_u^\eps)\, du,\,\,0\le s\le t,
\end{equation}
is a $d$-dimensional Brownian motion with respect to $\P^{\eps,\mu}$. Hence,
\begin{equation*}
dY_t^{(x)}=b_\eps(Y_t^{(x)},\mu_t^\eps)\,dt+dW_t^{\eps,\mu}.
\end{equation*}
Using this and \eqref{vspomogat}, we see that $P^t_\mu(\eps)\delta_x=\Law_{\P}(Z_t^\eps)=\Law_{\P^{\eps,\mu}}(Y_t^{(x)})$. By the Pinsker inequality (\cite{Pinsker}), we get
\begin{align*}
d_{TV}(P^t_\mu(\eps)
\delta_x,\Law_{\P}(Y_t^{(x)}))&=d_{TV}(\Law_{\P^{\eps,\mu}}(Y_t^{(x)}),\Law_{\P}(Y_t^{(x)}))\le
d_{TV}(\P^{\eps,\mu}, \P)\\
&\le \sqrt{2\E(\ln d\P/d\P^{\eps,\mu})}\le \eps D\sqrt{t},\quad t>0.
\end{align*}
Similarly, $d_{TV}(P^t_\mu(\eps) \delta_y,\Law_{\P}(Y_t^{(y)}))\le \eps D\sqrt{t}$. Thus, with the help of
\eqref{dtvest} we finally obtain
\begin{equation*}
d_{TV}(P^t_\mu(\eps) \delta_x,P^t_\mu(\eps) \delta_y)\le 2(1-\alpha(R,t))+2\eps D\sqrt{t}, \quad x,y\in S_V,
t>0.
\end{equation*}
This yields \eqref{Lyap}--\eqref{dobr} with $T=1$ and $\eps_0=\frac{\alpha(R,1)}{2D}  \wedge \frac {r}{2D}$.

\vspace{0.5ex} To complete the proof, it remains to note that estimate \eqref{est_mera} directly follows from \eqref{Vest}:
\begin{equation*}
(P^{t}_\mu(\eps)\mu)(V)=\int_E V(x)\,(P^{t}_\mu(\eps)\mu)(dx)=\int_E P^{t}_\mu(\eps)V(x)\,\mu(dx)\le
\mu(V)+K.\qedhere
\end{equation*}
\end{proof}

\begin{Lemma}\label{L:svoistvo_ravnomerki}
Suppose the conditions of Theorem~$\ref{Th:mv}$ are satisfied. If a measure $\pi$ is invariant for nonlinear operator $P_{\sbt}^T(\eps) \sbt$ (that is $P_\pi^T(\eps) \pi=\pi$), then $\pi(V)\le K$.

Here $\eps\in[0,\eps_0]$; the constants $T$, $\eps_0$, and $K$ are the same as in Lemma~$\ref{L:3Proverka}$.
\end{Lemma}
\begin{proof}
Fix $S>0$. Since the measure $\pi$ is invariant, we have $P^{nT}_{\pi}(\eps)\pi=\pi$ for any positive integer $n$. Therefore, using \eqref{Vest}, we derive
\begin{align*}
\int_{\mathbb{R}^d} (V(x)\wedge S)\,\pi(dx)&=\int_{\mathbb{R}^d} P^{nT}_{\pi}(\eps)(V(x)\wedge
S)\,\pi (dx)\\
&\le \int_{\mathbb{R}^d} (P^{nT}_{\pi}(\eps)V(x)\wedge S)\,\pi(dx)\\
&\le \int_{\mathbb{R}^d} ((e^{-\kappa nT r/4}V(x)+K)\wedge S)\,\pi (dx)\\
&\le K+\int_{\mathbb{R}^d} (e^{-\kappa nT r/4}V(x)\wedge S)\,\pi (dx),
\end{align*}
where $\kappa=r/4\wedge 1$. By  Lebesgue's dominated convergence theorem, the integral in the right-hand side of the above inequality tends to $0$ as $n\!\to\infty$. Hence
$$
\int_{\mathbb{R}^d} (V(x)\wedge S)\,\pi(dx)\!\le\! K.
$$
By taking the limit as $S\to\infty$ and applying Fatou's lemma, we obtain $\pi(V)\le K$.
\end{proof}

\begin{Lemma}\label{L:3}
Suppose the conditions of Theorem~$\ref{Th:mv}$ are satisfied. Let $(X^{\eps,\mu}_t,\mu_t^\eps)_{t\ge0}$ and
$(X^{\eps,\nu}_t,\nu_t^\eps)_{t\ge0}$ be the strong solutions of SMVE \eqref{SMV} with initial conditions distributed as $\mu_0$ and $\nu_0$, respectively. Then
\begin{equation}\label{distmu}
d_{TV}(\mu_t^\eps,\nu_t^\eps)\le\sqrt 2   d_{TV}(\mu_0,\nu_0)e^{4\eps^2L^2 t},\quad t\ge 0.
\end{equation}
\end{Lemma}
\begin{proof}
As in the proof of Theorem~\ref{Th:DiscreteTime}, consider the measure $\eta(dx):=(d\mu_0/d\nu_0\wedge 1)\nu_0(dx)$, where $d\mu_0/d\nu_0$ is a  Radon--Nikodym derivative of absolutely continuous part of $\mu_0$
 with respect to $\nu_0$. Then
\begin{align}\label{l4mainest}
d_{TV}(\mu_t^\eps,\nu_t^\eps)&=d_{TV}(P_{\mu_0}^t(\eps)\mu_0,P_{\nu_0}^t(\eps)\nu_0)\nonumber\\
&\le d_{TV}(P_{\mu_0}^t(\eps)\eta,P_{\nu_0}^t(\eps)\eta) +
d_{TV}(P_{\mu_0}^t(\eps)(\mu_0-\eta),P_{\nu_0}^t(\eps)(\nu_0-\eta))\nonumber\\
&\le(2-d_{TV}(\mu_0,\nu_0))\sup_{x\in\mathbb{R}^d} d_{TV}(P_{\mu_0}^t(\eps)\delta_x,P_{\nu_0}^t(\eps)\delta_x)+
d_{TV}(\mu_0,\nu_0),
\end{align}
where in the last inequality we used $\eta(\mathbb{R}^d)=1-d_{TV}(\mu_0,\nu_0)/2$.

Using an argument close to that of the proof of Lemma~\ref{L:3Proverka}, let us estimate the total variation distance $d_{TV}(P_{\mu_0}^t(\eps)\delta_x,P_{\nu_0}^t(\eps)\delta_x)$. Fix $x\in\mathbb{R}^d$. Let
$(Y_t^{(x)})_{t\ge0}$ be the strong solution of SDE \eqref{sdu_igr}. Define a measure $\P^{\eps,\mu}$ on
$(\Omega,\mathcal{F})$ by formula \eqref{mera} and introduce the measure $\P^{\eps,\nu}$ in a similar way (with the corresponding substitution $\nu_s^\eps$ for $\mu_s^\eps$). By the Girsanov theorem, the process $(W_s^{\eps,\mu})_{0\le s\le t}$ defined by \eqref{newbd} and the process $(W_s^{\eps,\nu})_{0\le s\le t}$ defined similarly are $d$-dimensional Brownian motions with respect to the measures $\P^{\eps,\mu}$ and $\P^{\eps,\nu}$, correspondingly. As in the proof of Lemma~\ref{L:3Proverka}, we see that
$P_{\mu_0}^t(\eps)\delta_x=\Law_{\P^{\eps,\mu}}(Y_t^{(x)})$ and $P_{\nu_0}^t(\eps)\delta_x=\Law_{\P^{\eps,\nu}}(Y_t^{(x)})$.

Denote $\rho:=d\P^{\eps,\mu}/d\P^{\eps,\nu}$. By the Pinsker inequality, we get
\begin{align*}
d_{TV} &(P_{\mu_0}^t(\eps)\delta_x,P_{\nu_0}^t(\eps)\delta_x)=
d_{TV} (\Law_{\P^{\eps,\mu}}(Y_t^{(x)}),\Law_{\P^{\eps,\nu}}(Y_t^{(x)}))\le d_{TV} (\P^{\eps,\mu},\P^{\eps,\nu})\\
&\le \sqrt{2 \E^{\P^{\eps,\mu}} (\ln \rho)}=\eps\Bigl(\E^{\P^{\eps,\mu}} \int_0^t  \bigl|b_2(Y_s^{(x)},\mu^\eps_s)-b_2(Y_s^{(x)},\nu^\eps_s)\bigr|^2 \,ds \Bigr)^{1/2}\\
&\le \eps L \Bigl( \int_0^t  \bigl(d_{TV}(\mu^\eps_s, \nu^\eps_s)\bigr)^2 \,ds \Bigr)^{1/2}.
\end{align*}
Note that the right-hand side of this inequality does not depend on $x$. Combining this estimate with \eqref{l4mainest}, and using inequality $(a+b)^2\le 2 a^2+2b^2$, which holds for all real $a,b$, we obtain
\vspace{-2ex}
\begin{equation*}
(d_{TV}(\mu^\eps_t,\nu^\eps_t))^2\le 2 (d_{TV}(\mu_0,\nu_0))^2+2\eps^2L^2 (2-d_{TV}(\mu_0,\nu_0))^2 \int_0^t
\bigl(d_{TV}(\mu^\eps_s, \nu^\eps_s)\bigr)^2 \,ds.
\end{equation*}
Now the application of Gronwall's lemma to the function $\psi(t):=(d_{TV}(\mu^\eps_t,\nu^\eps_t))^2$ yields \eqref{distmu}.
\end{proof}

\begin{Lemma}\label{L:5estimate}
Suppose the conditions of Theorem~$\ref{Th:mv}$ are satisfied. Let $\eps\in[0,\eps_0]$, where the constant  $\eps_0$ is the same as in Lemma~$\ref{L:3Proverka}$. Let $\mu_0, \nu_0, \zeta\in\mathcal{P}(\mathbb{R}^d)$. Then
\begin{equation}\label{mainstatementl5}
d_{1+\beta V} \bigl(P^T_{\mu_0}(\eps) \zeta,P^T_{\nu_0}(\eps)\zeta\bigr) \le C\eps (1+\beta)
(1+\zeta(V))d_{1+\beta V} (\mu_0,\nu_0),
\end{equation}
where $\beta\ge0$; $T$ is the same as in Lemma~$\ref{L:3Proverka}$; and $C>0$ depends only on $L,T$ and $\eps_0$.
\end{Lemma}

\begin{proof}
We begin by observing that
\begin{equation}\label{ubralizetu}
d_{1+\beta V} \bigl(P^T_{\mu_0}(\eps) \zeta,P^T_{\nu_0}(\eps)\zeta\bigr)\le \int_{\mathbb{R}^d}d_{1+\beta V}
\bigl(P^T_{\mu_0}(\eps) \delta_x,P^T_{\nu_0}(\eps)\delta_x\bigr)\,\zeta(dx).
\end{equation}
Fix $x\in\mathbb{R}^d$. Arguing as in the proof of Lemma~\ref{L:3} and using the same notation, we see that ${P^T_{\mu_0}(\eps) \delta_x=\Law_{\P^{\eps,\mu}}}(Y_T^{(x)})$ and $P^T_{\nu_0}(\eps) \delta_x=\Law_{\P^{\eps,\nu}}(Y_T^{(x)})$. Therefore
\begin{align}\label{nach_dok}
d_{1+\beta V} \bigl(P^T_{\mu_0}(\eps) \delta_x,P^T_{\nu_0}(\eps)\delta_x\bigr)&=
d_{1+\beta V} (\Law_{\P^{\eps,\mu}}(Y_T^{(x)}),\Law_{\P^{\eps,\nu}}(Y_T^{(x)}))\nonumber\\
&=\sup_ {f:\,\, |f|\le1+\beta V}
\bigl(E^{\P^{\eps,\mu}}f(Y_T^{(x)})-E^{\P^{\eps,\nu}}f(Y_T^{(x)})\bigr)\nonumber\\
&=\sup_ {f:\,\, |f|\le1+\beta V} E^{\P^{\eps,\nu}}f(Y_T^{(x)})(\rho-1)\nonumber\\
&\le \bigl(\E^{\P^{\eps,\nu}} (1+\beta
V(Y_T^{(x)}))^2\bigr)^{1/2}\,\bigl(\E^{\P^{\eps,\nu}}(\rho-1)^2\bigr)^{1/2},
\end{align}
where $\rho:=d\P^{\eps,\mu}/d\P^{\eps,\nu}$. It follows from  \eqref{Lyap2} that
\begin{align}\label{firstterm}
\E^{\P^{\eps,\nu}} (1+\beta V(Y_T^{(x)}))^2 \le 2 + 2\beta^2 P^T_{\nu_0}(\eps) V(x)^2  \le 2\beta^2\gamma^2
V(x)^2 +2 +2\beta^2K^2.
\end{align}
To estimate the second factor in the right-hand side of \eqref{nach_dok}, we use the same technique as in \cite{MisV}. Denote $\phi(s):=b_2(Y_s^{(x)},\mu^\eps_s)- b_2(Y_s^{(x)},\nu^\eps_s)$. Then, we have
\begin{align*}
\E^{\P^{\eps,\nu}} \rho^2&=\E^{\P^{\eps,\nu}}\exp\Bigl(2\eps\int_0^T
\phi(s) \, dW_s^{\eps,\nu}-\eps^2 \int_0^T  |\phi(s)|^2\,ds\Bigr)\\
&\le \E^{\P^{\eps,\nu}}\exp\Bigl(6\eps^2 \int_0^T  |\phi(s)|^2\,ds\Bigr)\\
&\le \exp\Bigl(6\eps^2 L^2\int_0^T  \bigl(d_{TV}(\mu_s^\eps,\nu_s^\eps)\bigr)^2\,ds\Bigr)\\
&\le \exp\Bigl(12 \eps^2 L^2 T (d_{TV}(\mu_0,\nu_0))^2e^{8\eps^2L^2 T}\Bigr),
\end{align*}
where the second inequality follows from the Cauchy--Schwarz inequality; and in the last inequality we used
Lemma~\ref{L:3}. Since $\eps\le\eps_0$ and $e^a-1\le a e^a$ for any real $a$, we derive
\begin{align*}
\bigl(\E^{\P^{\eps,\nu}}(\rho-1)^2\bigr)^{1/2}\le 4\eps L \sqrt{T}
d_{TV}(\mu_0,\nu_0)\exp\{4\eps_0^2L^2 T(1+6e^{8\eps_0^2L^2 T})\}.
\end{align*}
Combining this with \eqref{nach_dok} and \eqref{firstterm}, we finally get
\begin{equation*}
d_{1+\beta V} \bigl(P^T_{\mu_0}(\eps) \delta_x,P^T_{\nu_0}(\eps)\delta_x\bigr) \le C\eps (1+\beta)
(1+V(x))d_{TV}(\mu_0,\nu_0),
\end{equation*}
where $C>0$ depends only on $L$, $T$ and $\eps_0$. This, together with \eqref{ubralizetu}, implies
\eqref{mainstatementl5}.
\end{proof}

\begin{proof}[Proof of Theorem~$\ref{Th:mv}$]
First of all, we observe that by Lemmas~\ref{L:2HM}~and~\ref{L:3Proverka} there exist $\beta>0$ and $\lambda\in[0,1)$ such that
\begin{equation*}
d_{1+\beta V}(P^T_\zeta (\eps)\mu ,P^T_\zeta (\eps)\nu )\le \lambda d_{1+\beta V}(\mu, \nu)
\end{equation*}
for any $\mu,\nu,\zeta\in\mathcal{P}(\mathbb{R}^d)$ and $\eps\in[0,\eps_0]$, where $\eps_0$ and $T$ are the same as in Lemma~\ref{L:3Proverka}. Therefore, Lemma~\ref{L:5estimate} yields
\begin{align*}
d_{1+\beta V}(P^T_\mu  (\eps)\mu,P^T_\nu  (\eps) \nu)&\le d_{1+\beta V}(P^T_\mu  (\eps) \mu,P^T_\mu
(\eps)\nu)+d_{1+\beta V}(P^T_\mu (\eps) \nu,P^T_\nu (\eps)\nu)\\
&\le(C\eps (1+\beta) (1+\nu(V))+ \lambda )d_{1+\beta V}(\mu, \nu)
\end{align*}
for any measures $\mu,\nu\in\mathcal{P}(\mathbb{R}^d)$. Iteratively applying this inequality $n$ times and taking into account \eqref{est_mera}, we get
\begin{align}\label{wdtv_okon}
d_{1+\beta V}(P^{nT}_\mu  (\eps)\mu,P^{nT}_\nu  (\eps) \nu)&\le (C\eps (1+\beta) (1+\max_{0\le i\le n} V_i)+
\lambda )^n d_{1+\beta V}(\mu, \nu)\nonumber\\
&\le \theta(\eps,\nu)^n  d_{1+\beta V}(\mu, \nu),
\end{align}
where we denoted $V_i:=(P^{iT}_\nu(\eps)\nu)(V)$ and $\theta(\eps,\nu):=\lambda+C\eps (1+\beta) (1+K+\nu(V))$.

Consider a measure $\nu_0\in\mathcal{P}(\mathbb{R}^d)$ such that $\int_{\mathbb{R}^d} e^x \,\nu_0(dx)<\infty$. It follows from the definition of the function $V$ that $\nu_0(V)<\infty$. Let us take ``small'' $\eps_1\in[0,\eps_0]$ such that
\begin{equation*}
\lambda+C\eps_1 (1+\beta) (1+K+(\nu_0(V)\vee K))<1.
\end{equation*}
It is possible to find such $\eps_1$ because $\lambda<1$. Let us prove that for any $0\le\eps\le\eps_1$ the strong solution of SMVE \eqref{SMV} has a unique invariant measure.

Denote by $\mathcal{P}_V(\mathbb{R}^d)$ the space of all probability measures on $(\mathbb{R}^d,\mathcal{B}(\mathbb{R}^d))$ which integrate $V$. Let $(X^{\eps,\nu}_t,\nu_t^\eps)_{t\ge0}$ be the strong solution of \eqref{SMV} with initial condition distributed as $\nu_0$. Consider the sequence of measures $(\nu^\eps_{nT})_{n\in\mathbb{Z}_+}$, where $\eps\in[0,\eps_1]$. We claim that this sequence is a Cauchy sequence in the metric space $(\mathcal{P}_V(\mathbb{R}^d),d_{1+\beta V})$. Indeed, for any
$m,n\in\mathbb{Z}_+$ we have
\begin{align*}
d_{1+\beta V}(\nu^\eps_{(n+m)T},\nu^\eps_{nT})&= d_{1+\beta V} (P^{nT}_{\nu_{mT}^\eps}(\eps)\nu_{mT}^\eps,
P^{nT}_{\nu_0}(\eps)\nu_0)\\
&\le \theta(\eps,\nu_0)^n d_{1+\beta V} (\nu_0, \nu_{mT}^\eps)\\
&\le \theta(\eps,\nu_0)^n (2+\beta\nu_0(V)+\beta\nu_{mT}^\eps (V)) \le C_1 \theta(\eps,\nu_0)^n,
\end{align*}
where $C_1=2+2\beta \nu_0 (V)+\beta K$; in the second inequality we applied \eqref{wdtv_okon}; and in the last inequality we used \eqref{est_mera}. Since $\theta(\eps,\nu_0)\!\le\theta(\eps_1,\nu_0)\!<\!1$, we get ${d_{1+\beta V}(\nu^\eps_{(n+m)T},\nu^\eps_{nT})\!\to0}$ as $n,m\to\infty$.

The space $(\mathcal{P}_V(\mathbb{R}^d),d_{1+\beta V})$ is complete; hence there exists a measure
$\pi^\eps\in \mathcal{P}_V(\mathbb{R}^d)$ such that $d_{1+\beta V}(\nu^\eps_{nT},\pi^\eps)\to 0$ as $n\to\infty$. Arguing as in the proof of Theorem~\ref{Th:DiscreteTime}(i) and applying Lemma~\ref{L:svoistvo_ravnomerki}, we see that the measure $\pi^\eps$ is a unique invariant measure
of the nonlinear operator $P_{\sbt}^T(\eps) \sbt$. In other words,
$P_{\pi^\eps}^T(\eps) \pi^\eps=\pi^\eps$ and if $P_{\nu}^T(\eps) \nu=\nu$ for a measure $\nu\in
\mathcal{P}(\mathbb{R}^d)$, then $\nu=\pi^\eps$.

Let us verify that the measure $\pi^\eps$ is invariant for solutions of SMVE \eqref{SMV}. To do this it is sufficient to check that for any $t\ge0$ the measures $\pi^\eps_t:=P_{\pi^\eps}^t(\eps) \pi^\eps$ and $\pi^\eps$ are equal. Assume the converse. Let $\pi^\eps_t\neq \pi^\eps$ for some  $t>0$. Since $\pi_T^\eps=P_{\pi^\eps}^T(\eps) \pi^\eps=\pi^\eps$, we derive
\begin{equation*}
P_{\pi^\eps_t}^T(\eps) \pi^\eps_t=P_{\pi^\eps}^{T+t}(\eps) \pi^\eps=P_{\pi^\eps_T}^t(\eps) \pi^\eps_T=
P_{\pi^\eps}^t(\eps) \pi^\eps=\pi^\eps_t.
\end{equation*}
Consequently, the nonlinear operator $P_{\sbt}^T(\eps) \sbt${ } has two different invariant measures (namely,
$\pi^\eps$ and $\pi^\eps_t$). By the above, this is impossible. Hence, $\pi^\eps_t= \pi^\eps$. Thus, the measure $\pi^\eps$ is a unique invariant measure of \eqref{SMV}.

Finally let us establish the convergence rate \eqref{otsenka}. Consider a measure $\mu_0\in\mathcal{P}(\mathbb{R}^d)$ such that $I(\mu_0)=\int_{\mathbb{R}^d} e^x \,\mu_0(dx)<\infty$. Let 
$(X^{\eps,\mu}_t,\mu_t^\eps)_{t\ge0}$ be a strong solution of SMVE \eqref{SMV} with initial condition distributed as $\mu_0$. We make use of \eqref{wdtv_okon} to obtain for any $t>0$
\begin{align*}
d_{TV}(\mu_t^\eps,\pi^\eps)&\le d_{1+\beta V}(\mu_t^\eps,\pi^\eps)= d_{1+\beta V}
(P^{[t/T]T}_{\mu_{\{t/T\}T}^\eps} (\eps) \mu_{\{t/T\}T}^\eps, P^{[t/T]T}_{\pi^\eps}(\eps)\pi^\eps)\\
&\le \theta(\eps,\pi^\eps)^{[t/T]} d_{1+\beta V}(\mu_{\{t/T\}T}^\eps,\pi^\eps)\\
&\le \theta(\eps,\pi^\eps)^{t/T-1}
(2+\beta I(\mu_0)+(1+\beta) K).
\end{align*}
Here $[\cdot]$ and $\{\cdot\}$  are the fractional and integer parts of a real number, respectively. To complete the proof, it remains to note that Lemma~\ref{L:svoistvo_ravnomerki} implies $\pi^\eps(V)\le K$. Since $\eps\le \eps_1$, we see that $\theta(\eps,\pi^\eps)<1$. This yields \eqref{otsenka}.
\end{proof}

\bigskip

\textbf{Acknowledgements}. The author is grateful to Professor A.V.~Bulinski and Professor A.Yu.~Veretennikov
for posing the problem, help, and constant attention to this work. The author also would like to thank Professor A.M. Kulik for useful discussions. This paper was partially written during the author's stay at ICMS --- International Centre for Mathematical Sciences (Edinburgh, UK). The author is grateful to ICMS for their support and hospitality.


\newpage
\begin{thebibliography}{99}


\bibitem{myself}
\textsc{O.A. Butkovsky} (2012). On the convergence of nonlinear Markov chains. \textit{Doklady Mathematics},
\textbf{86}(3), 824--826.



\bibitem{CGM}
\textsc{P. Cattiaux}, \textsc{A. Guillin}, \textsc{F. Malrieu} (2008). Probabilistic approach for granular media equations in the non-uniformly convex case. \textit{Prob. Theory Rel. Fields}, \textbf{140}(1), 19--40.

\bibitem{Dobr}
\textsc{R.L. Dobrushin} (1956). Central limit theorem for nonstationary Markov chains. I. \textit{Theory
Probab. Appl.}, \textbf{1}(1), 65--80.


\bibitem{Fr}
\textsc{T.D. Frank} (2004). Stochastic Feedback, Nonlinear Families of Markov processes, and Nonlinear Fokker--Planck
Equations. \textit{Physica A}, \textbf{331}, 391--408.

\bibitem{Ganz}
\textsc{A. Ganz} (2008). \textit{Approximation of equilibrium distributions of some stochastic systems with McKean--Vlasov
interactions}. Ph.D. Thesis. Universit\'{e} de Nice.

\bibitem{Hair}
\textsc{M. Hairer} (2010). \textit{Convergence of Markov processes}. Lecture Notes, University of Warwick.
Available at http://www.hairer.org/notes/Convergence.pdf.

\bibitem{HM}
\textsc{M. Hairer}, \textsc{J.C. Mattingly} (2011). Yet another look at Harris' ergodic theorem for Markov
chains. \textit{Seminar on Stochastic Analysis, Random Fields and Applications~VI}. Progress in Probability,
\textbf{63}, 109--117.

\bibitem{JMW}
\textsc{B. Jourdain}, \textsc{S. M\'{e}l\'{e}ard}, \textsc{W.A. Woyczynski} (2008). Nonlinear SDEs driven by L\'{e}vy processes and related PDEs. \textit{ALEA Lat. Am. J. Probab. Math. Stat.}, \textbf{4}, 1--29.

\bibitem{Kol}
\textsc{V.N. Kolokoltsov} (2010). \textit{Nonlinear Markov processes and kinetic equations}. Cambridge Tracts in Mathematics, \textbf{182}. Cambridge: Cambridge Univ. Press.

\bibitem{Kulik}
\textsc{A.M.~Kulik} (2009). Exponential ergodicity of the solutions to SDE's with a jump noise. \textit{Stoch. Process. Appl.}, \textbf{119}(2), 602-632.

\bibitem{Mc}
\textsc{H.P. McKean} (1966). A class of Markov processes associated with nonlinear parabolic equations. \textit{Proc. Natl.
Acad. Sci. USA}, \textbf{56}(6), 1907--1911.

\bibitem{MT}
\textsc{S. Meyn}, \textsc{R.L. Tweedie} (2009). \textit{Markov Chains and Stochastic Stability}, 2nd Edn.
N.Y.: Cambridge Univ. Press.

\bibitem{MisV}
\textsc{Yu.S. Mishura}, \textsc{A.Yu. Veretennikov} (2013). Existence and uniqueness theorems for solutions of
McKean--Vlasov stochastic equations. Preprint.

\bibitem{MV}
\textsc{S.A. Muzychka}, \textsc{K.L. Vaninsky} (2011). A class of nonlinear random walks related to the
Ornstein-Uhlenbeck process. \textit{Markov Process. Related Fields}, \textbf{17}, 277--304.

\bibitem{Pinsker}
\textsc{M.S. Pinsker} (1964). \textit{Information and Information Stability of Random Variables and
Processes}. San Francisco: Holden-Day.

\bibitem{Szn}
\textsc{A.-S. Sznitman} (1991). Topics in propagation of chaos. In: \textit{\'{E}c. \'{E}t\'{e} Probab. St.-Flour XIX}, Lecture Notes in Math, \textbf{1464}, 165--251, Berlin: Springer.

\bibitem{V1987}
\textsc{A.Yu. Veretennikov} (1987). Bounds for the mixing rate in the theory of stochastic equations. \textit{Theory  Probab.  Appl.}, \textbf{32}(2), 273--281.

\bibitem{VerTVP}
\textsc{A.Yu. Veretennikov} (2000). On polynomial mixing and convergence rate for stochastic difference and differential equations. \textit{Theory  Probab.  Appl.}, \textbf{44}(2), 361-374.

\bibitem{Ver}
\textsc{A.Yu. Veretennikov} (2006). On ergodic measures for McKean--Vlasov stochastic equations. In: \textit{Monte Carlo and Quasi-Monte Carlo Methods}, 471--486, Berlin: Springer.

\bibitem{Wang}
\textsc{F.-Y. Wang} (2011). Harnack inequality for SDE with multiplicative noise and extension to Neumann semigroup on
nonconvex manifolds. \textit{Ann. Probab.}, \textbf{39}(4), 1449--1467.

\end{thebibliography}
\end{document}